\documentclass[a4paper]{article}
\usepackage[english]{babel}
\usepackage[latin1]{inputenc}
\usepackage{amsmath,amsthm,amssymb,amsfonts}
\usepackage{url}
\usepackage{esvect}
\usepackage{authblk}
\usepackage[backref=page]{hyperref}
\hypersetup{pdfpagemode=UseNone, pdfstartview={FitH}}
\usepackage{cleveref}
\newtheorem{theorem}{Theorem}
\newtheorem{remark}{Remark}
\newtheorem{problem}{Problem}
\newtheorem{definition}{Definition}

\newtheorem{lemma}{Lemma} 
\newtheorem{corollary}{Corollary}
\crefname{problem}{Problem}{Problems}
\Crefname{problem}{Problem}{Problems}
\newcommand{\lag}{\left \langle}
\newcommand{\rog}{\right \rangle}
\newcommand{\df}{\stackrel{\text{\tiny def}}{=}}
\crefname{problem}{Problem}{Problems}
\Crefname{problem}{Problem}{Problems}

\begin{document}
\date{}
\title{Trakhtenbrot theorem and first-order axiomatic extensions of MTL}
\author{Matteo Bianchi}
\affil{Department of Computer Science, Università degli Studi di Milano, Via Comelico 39/41, 20135, Milano, Italy \\\url{matteo.bianchi@unimi.it}}
\maketitle
\begin{abstract}
In 1950, B.A. Trakhtenbrot showed that the set of first-order tautologies associated to finite models is not recursively enumerable. In 1999, P. H\'ajek generalized this result to the first-order versions of \L ukasiewicz, G\"odel and Product logics. In this paper we extend the analysis to the first-order axiomatic extensions of MTL. Our main result is the following. Let L be an axiomatic extension L of MTL s.t. TAUT$_\text{L}$ is decidable, and whose corresponding variety is generated by a chain: for every generic L-chain $\mathcal{A}$ the set fTAUT$^\mathcal{A}_{\forall}$ (the set of first-order tautologies associated to the finite $\mathcal{A}$-models) is $\Pi_1$. Moreover, if in addition L is an extension of BL or an extension of SMTL or an extension of WNM, then for every generic L-chain $\mathcal{A}$ the set fTAUT$^\mathcal{A}_{\forall}$ is $\Pi_1$-complete. More in general, for every axiomatic extension L of MTL s.t. TAUT$_\text{L}$ is decidable there is no L-chain $\mathcal{A}$ such that L$\forall$ is complete w.r.t. the class of finite $\mathcal{A}$-models. We have negative results also if we expand the language with the $\Delta$ operator.
\end{abstract}
\section{Introduction and motivations}
In \cite{trak1}, B.A. Trakhtenbrot showed that the set of first-order tautologies associated to finite models is not recursively enumerable, in classical first-order logic: moreover, it is known that such set is $\Pi_1$-complete (in \cite{vau,bgg} it is shown that the theorem works also with languages containing at least a binary predicate, and without equality). This result implies the fact that the completeness w.r.t. finite models does not hold, in first-order logic (indeed, the set of theorems of classical predicate logic is $\Sigma_1$-complete). One can ask if a similar result holds also in non-classical logics, for example many-valued logics. A first answer was given in \cite{trak} by P. H\'ajek, that generalized Trakhtenbrot theorem to the first-order versions of \L ukasiewicz, G\"odel and Product logics (with respect to their standard algebras): that paper was published in 1999, and from then a much larger family of many-valued logics has been introduced, in particular the monoidal t-norm based logic MTL and its extensions (\cite{eg,hand}). These logics extend the well known full Lambek calculus, and they are all algebraizable in the sense of \cite{bp}: in particular, the semantics related to each logic forms an algebraic variety.

Differently to what happens in classical logic, in these many-valued logics we do not have necessarily a single totally ordered algebraic structure in which we can evaluate the truth-values of a formula: in particular, if L is an axiomatic extension of MTL, the existence of an L-chain w.r.t. the logic is complete to is called single chain completeness (SCC). Not all the axiomatic extensions of MTL enjoy this property: in \cite{ssc1} an extensive study has been done, about the SCC.	

For every axiomatic extension L of MTL, we have a completeness theorem w.r.t. the class of L-algebras. In the first-order case, however, we need to restrict to totally ordered algebras: indeed, if not, the soundness does not necessarily holds, see \cite[Example 5.4]{eghm} for a counterexample over G\"odel logic. This is not by chance, but it is a consequence of the fact that such logics are axiomatized in the way to have the completeness w.r.t. the class of all chains (such development of first-order logics has many connections with the works of Mostowski and Rasiowa, as explained in \cite{hmr}). So, here the analysis of single chain completeness becomes even more justified, than in the propositional case. However, such a study is also (much) harder than in the propositional case, as pointed out in \cite{ssc1}.

In this article we show a generalized version of Trakhtenbrot theorem. Our main result is the following. Let L be an axiomatic extension L of MTL s.t. TAUT$_\text{L}$ is decidable, and whose corresponding variety is generated by a chain: for every generic L-chain $\mathcal{A}$ the set fTAUT$^\mathcal{A}_{\forall}$ (the set of first-order tautologies associated to the finite $\mathcal{A}$-models) is $\Pi_1$. Moreover, if in addition L is an extension of BL or an extension of SMTL or an extension of WNM, then for every generic L-chain $\mathcal{A}$ the set fTAUT$^\mathcal{A}_{\forall}$ is $\Pi_1$-complete.  
As a corollary, we have that if L is one of BL, BL$_n$, {\L}, {\L}$_n$, G, G$_n$, $\Pi$, SMTL, SBL, SBL$^n$, SBL$_n$,WNM, NM, NMG, RDP, DP, and $\mathcal{A}$ is a generic L-chain, then fTAUT$^\mathcal{A}_{\text{L}\forall}$ is $\Pi_1$-complete. 

We also show that, for every axiomatic extension L of MTL s.t. TAUT$_\text{L}$ is decidable, there is no L-chain $\mathcal{A}$ such that L$\forall$ is complete w.r.t. the class of finite $\mathcal{A}$-models. So, the (first-order) single chain completeness w.r.t. finite models fails to hold.

We conclude by discussing the expansions with the $\Delta$ operator.
\section{Some basic background}
We assume that the reader is familiar with monoidal t-norm based logics and its extensions, in the propositional and in the first-order case. For a reference, see \cite{hand,haj,eg,ch}.
\subsection{Syntax}
The language of MTL is based over the connectives $\{\land,\&,\to,\bot\}$: the formulas are built in the usual inductive way from these connectives, and a denumerable set of variables.

Useful derived connectives are the following:
\begin{align}
\tag{negation}\neg\varphi\df& \varphi\to\bot\\
\tag{disjunction}\varphi\vee\psi\df& ((\varphi\to\psi)\to\psi)\land((\psi\to\varphi)\to\varphi)\\
\tag{biconditional}\varphi\leftrightarrow\psi\df&(\varphi\to\psi)\land(\psi\to\varphi)
\end{align}
MTL can be axiomatized with a Hilbert style calculus: for the reader's convenience, we list the axioms of MTL:
\begin{align}
\tag{A1}&(\varphi \rightarrow \psi)\rightarrow ((\psi\rightarrow \chi)\rightarrow(\varphi\rightarrow \chi))\\
\tag{A2}&(\varphi\&\psi)\rightarrow \varphi\\
\tag{A3}&(\varphi\&\psi)\rightarrow(\psi\&\varphi)\\
\tag{A4}&(\varphi\land\psi)\rightarrow \varphi\\
\tag{A5}&(\varphi\land\psi)\rightarrow(\psi\land\varphi)\\
\tag{A6}&(\varphi\&(\varphi\rightarrow \psi))\rightarrow (\psi\land\varphi)\\
\tag{A7a}&(\varphi\rightarrow(\psi\rightarrow\chi))\rightarrow((\varphi\&\psi)\rightarrow \chi)\\
\tag{A7b}&((\varphi\&\psi)\rightarrow \chi)\rightarrow(\varphi\rightarrow(\psi\rightarrow\chi))\\
\tag{A8}&((\varphi\rightarrow\psi)\rightarrow\chi)\rightarrow(((\psi\rightarrow\varphi)\rightarrow\chi)\rightarrow\chi)\\
\tag{A9}&\bot\rightarrow\varphi
\end{align}
As inference rule we have modus ponens:
\begin{equation}
\tag{MP}\frac{\varphi\quad \varphi\rightarrow\psi}{\psi}
\end{equation}
An axiomatic extension of MTL is a logic obtained by adding one or more axiom schemata to it. A theory is a set of formulas: the notion of proof and logical consequence are defined as in the classical case.

In this paper we focus on some extensions of MTL: in particular, BL, $\Pi$, {\L}, {\L}$_n$, BL$_n$, SMTL, SBL, SBL$^n$, SBL$_n$, WNM, RDP, NMG, G, G$_n$, NM, DP\footnote{This logic was introduced in \cite{nog,hnp}, and called S$_3$MTL. In \cite{abv} it has been further analysed under the name DP, for it is the logic of drastic product chains.} (see \cite{haj,prod,mun,grig,bln,rat,sbl,eg,rdp,nog,hnp,abv,ciaeg,dum,dm} for details). The logics WNM, G, DP, SMTL are axiomatized as MTL plus, respectively:\footnote{The notation $\varphi^n$ indicates $\underbrace{\varphi\&\dots\&\varphi}_{n\mbox{ }\text{times}}$.}
\begin{align*}
\tag*{(wnm)}&\neg(\varphi\&\psi)\vee((\varphi\land\psi)\to(\varphi\&\psi)).\\
\tag*{(id)}&\varphi\to(\varphi\&\varphi).\\
\tag*{(dp)}&\varphi\vee\neg(\varphi\&\varphi).\\
\tag*{(s)}&\neg(\neg\varphi\land\varphi).
\end{align*}
For $n\geq2$, G$_n$ is axiomatized as G plus:
\begin{equation*}
\tag*{(g$_n$)}\bigvee_{i<n}(x_i\to x_{i+1}).
\end{equation*}
RDP, NMG, NM are axiomatized as WNM plus, respectively\footnote{Usually, NMG is axiomatized as MTL plus $(\neg\neg \varphi\to \varphi)\vee ((\varphi\land \psi) \to (\varphi\&\psi))$. Here, we use the more compact axiomatization introduced in \cite{ab}.}:
\begin{align*}
\tag*{(rdp)}&(\varphi\to\neg\varphi)\vee\neg\neg \varphi.\\
\tag*{(nmg)}&(\neg\neg \varphi\to \varphi)\vee \neg\neg\varphi.\\
\tag*{(inv)}&\neg\neg \varphi\to \varphi.\label{inv}
\end{align*}
BL is axiomatized as MTL plus:
\begin{equation*}
\tag*{(div)}(\varphi\land\psi)\to(\varphi\&(\varphi\to\psi)).
\end{equation*}
SBL, $\Pi$, \L\, are axiomatized as BL plus, respectively:
\begin{align*}
\tag*{(s)}&\neg(\neg\varphi\land\varphi).\\
\tag*{(c)}&\neg\varphi\vee ((\varphi\to(\varphi\&\psi))\to\psi).\\
\tag*{(inv)}&\neg\neg \varphi\to \varphi.
\end{align*}
SBL$^n$ is axiomatized as SBL plus:
\begin{equation*}
\tag*{($c_n$)}\varphi^{n}\to\varphi^{n+1}.\label{eq:cn}
\end{equation*}
SBL$_n$ is axiomatized as SBL$^n$ plus the following set of axiom schemata.
\begin{equation*}
\tag*{($d_{n,m}$)}(\varphi^{m-1}\leftrightarrow(\varphi\to\varphi^n))^n\to\varphi^n.\label{eq:dn}
\end{equation*}
for every $m<n$ such that $m$ does not divide $n$.

BL$_n$ is axiomatized as BL plus \ref{eq:cn} and \ref{eq:dn}, for every $m<n$ such that $m$ does not divide $n$.

\L$_n$ is axiomatized as BL$_n$ plus \ref{inv}.
\subsection{Semantics}
An MTL-algebra is an algebra $\lag A,*,\Rightarrow,\sqcap,\sqcup,0,1\rog$ such that:
\begin{enumerate}
\item $\lag A,\sqcap,\sqcup, 0,1\rog$ is a bounded lattice with minimum $0$ and maximum $1$.
\item $\lag A,*,1 \rog$ is a commutative monoid.
\item $\lag *,\Rightarrow \rog$ forms a \emph{residuated pair}: $z*x\leq y$ iff $z\leq x\Rightarrow y$ for all $x,y,z\in A$.
\item The following axiom holds, for all $x,y\in A$:
\begin{equation}
\tag{Prelinearity}(x\Rightarrow y)\sqcup(y\Rightarrow x)=1
\end{equation}
A totally ordered MTL-algebra is called MTL-chain. An MTL-algebra is called \emph{standard} whenever its support is $[0,1]$: it is well known (see \cite{eg,beg}) that this is the case if and only if $*$ is a left-continuous t-norm (see \cite{kmp} for a monograph on t-norms).
\end{enumerate}
In the rest of the paper the notation $\sim x$ will denote $x\Rightarrow 0$.

Given an MTL-chain $\mathcal{A}$, we define $A^+=\{x\in A:\, x>\sim x\}$. A negation fixpoint is an element such that $x=\sim x$: an easy check shows that, if an MTL-chain has a such element, then it is unique.

Let L be an axiomatic extension of MTL. It is known (see \cite{nog,dist}) that L is algebraisable in the sense of \cite{bp}, and that the equivalent algebraic semantics forms a subvariety of MTL-algebras, called L-algebras.	 We will denote by $\mathbb{L}$ such variety.
On the other hand, each subvariety $\mathbb{L}$ of $\mathbb{MTL}$ is algebraisable, and we will denote by L the corresponding axiomatic extension of MTL.

In particular L is the extension of MTL via a set of axioms $\{\varphi\}_{i \in I}$ if and only if $\mathbb{L}$ is the subvariety
of MTL-algebras satisfying $\{\bar{\varphi} = 1\}_{i \in I}$, where $\bar{\varphi}$ is obtained from $\varphi$ by replacing
each occurrence of $\&,\to,\land,\vee,\neg,\bot$ with $*,\Rightarrow,\sqcap,\sqcup,\sim,0$, and every formula symbol occurring in
$\varphi$ with an individual variable.

We recall that the standard MV-algebra $[0,1]_\text{\L}$ is an MTL-chain with $[0,1]$ as support, and such that, for every $x,y\in [0,1]$:
\begin{equation*}
x*y=\max\{0,x+y-1\}\qquad x\Rightarrow y=\min\{1,1-x+y\}.
\end{equation*}
Moreover every MV-chain of $n+1$ elements is isomorphic to the subalgebra of $[0,1]_\text{\L}$ having $\{0,\frac{1}{n},\dots,\frac{n-1}{n},1\}$ as support. This algebra will be called $\mathbf{L}_n$, and MV$_n$ its generated variety. It is known that (see \cite{grig,mun}) an MV-chain belong to MV$_n$ if and only if it is isomorphic to $\mathbf{L}_k$, with $k$ that divides $n$.

Moving to the case of WNM, we recall (see \cite{nog}) that in every WNM-chain the operations $*$ and $\Rightarrow$ have this form:
\begin{equation}\label{eq:op}
x*y=\begin{cases}
0&\text{if }x\leq \sim y\\
\min\{x,y\}&\text{otherwise.}
\end{cases}\qquad
x\Rightarrow y=\begin{cases}
1&\text{if }x\leq y\\
\max\{\sim x,y\}&\text{otherwise.}
\end{cases}
\end{equation}
In particular, an easy check shows that if $\mathcal{A}$ is a WNM-chain, then every element $x\in A^+$ is idempotent, i.e. $x*x=x$.

Finally, the notions of evaluation, tautology and completeness are defined in the usual way. 

Let L be an axiomatic extension of MTL: we say that an L-chain $\mathcal{A}$ is generic whenever the variety generated by $\mathcal{A}$ is $\mathbb{L}$, i.e. L is complete w.r.t. $\mathcal{A}$ (for every formula $\varphi$, $\vdash_\text{L}\varphi$ iff $\mathcal{A}\models\varphi$).
\subsection{First-order case}
In this section we briefly present the first-order versions of MTL and its axiomatic extensions: more details can be found in \cite{ch,hand}.
\begin{definition}\label{lang}
A first-order language is a \emph{countable} set $\mathbf{P}$ of predicate symbols, containing \emph{at least} a binary one (i.e. we do not work with monadic fragments). To simplify our analysis we overlook constant, function symbols, and we work without equality. We have the ``classical'' quantifiers $\forall,\exists$. The notions of term (note that our terms coincide with variables), formula, closed formula, term substitutable in a formula are defined like in the classical case (\cite{ch,hand}); the connectives are those of the propositional level.
\end{definition}
Let L be an axiomatic extension of MTL: then its first-order version, L$\forall$, is axiomatized as follows:
\begin{itemize}
\item The axioms resulting from the axioms of L by the substitution of the propositional variables by the first-order formulas.
\item The following axioms:
\begin{align}
\tag{$\forall 1$}&(\forall x)\varphi(x)\rightarrow \varphi(x/t)\text{( }t\text{ substitutable for }x\text{ in }\varphi(x)\text{)}\\
\tag{$\exists 1$}&\varphi(x/t)\rightarrow (\exists x)\varphi(x)\text{( }t\text{ substitutable for }x\text{ in }\varphi(x)\text{)}\\
\tag{$\forall 2$}&(\forall x)(\nu \rightarrow \varphi)\rightarrow (\nu \rightarrow (\forall x)\varphi)\text{ (}x\text{ not free in }\nu\text{)}\\
\tag{$\exists 2$}&(\forall x)(\varphi \rightarrow \nu)\rightarrow ((\exists x)\varphi\rightarrow \nu)\text{ (}x\text{ not free in }\nu\text{)}\\
\tag{$\forall 3$}&(\forall x)(\varphi \vee \nu)\rightarrow ((\forall x)\varphi \vee \nu)\text{ (}x\text{ not free in }\nu\text{)}
\end{align}
\end{itemize}
The rules of L$\forall$ are: Modus Ponens: $\frac{\varphi\quad\varphi\to\psi}{\psi}$ and Generalization: $\frac{\varphi}{(\forall x)\varphi}$.
\paragraph*{}
As regards to semantics, we need to restrict to L-chains: given an L-chain $\mathcal{A}$, a finite $\mathbf{A}$-model is a structure $\mathbf{M}=\lag M,\{r_P\}_{P\in \mathbf{P}}\rog$, where:
\begin{itemize}
\item M is a \emph{finite} non-empty set.
\item for each $P\in \mathbf{P}$ of arity\footnote{If $P$ has arity zero, then $r_P\in A$.} $n$, $r_P:M^n\to A$.
\end{itemize}
 For each evaluation over variables $v:Var\to M$, the truth value of a formula $\varphi$ ($\Vert\varphi\Vert_{\mathbf{M},v}^\mathcal{A}$) is defined inductively as follows:
\begin{itemize}
 \item $\Vert P(x_1,\dots,x_n)\Vert_{\mathbf{M},v}^\mathcal{A}=r_P(v(x_1),\dots,v(x_n))$.
\item The truth value commutes with the connectives of L$\forall$, i.e.
\begin{align*}
\Vert\varphi\rightarrow\psi\Vert^\mathcal{A}_{\mathbf{M},v}&=\Vert\varphi\Vert^\mathcal{A}_{\mathbf{M},v}\Rightarrow\Vert\psi\Vert^\mathcal{A}_{\mathbf{M},v}\\
\Vert\varphi\&\psi\Vert^\mathcal{A}_{\mathbf{M},v}&=\Vert\varphi\Vert^\mathcal{A}_{\mathbf{M},v}*\Vert\psi\Vert^\mathcal{A}_{\mathbf{M},v}\\
\Vert \bot \Vert^\mathcal{A}_{\mathbf{M},v}&=0\\ \Vert\varphi\land\psi\Vert^\mathcal{A}_{\mathbf{M},v}&=\Vert\varphi\Vert^\mathcal{A}_{\mathbf{M},v}\sqcap\Vert\psi\Vert^\mathcal{A}_{\mathbf{M},v}\\
\Vert\varphi\vee\psi\Vert^\mathcal{A}_{\mathbf{M},v}&=\Vert\varphi\Vert^\mathcal{A}_{\mathbf{M},v}\sqcup\Vert\psi\Vert^\mathcal{A}_{\mathbf{M},v}.
\end{align*}
\item $\Vert(\forall x)\varphi\Vert_{\mathbf{M},v}^\mathcal{A}=\min\{\Vert\varphi\Vert_{M,v'}^\mathcal{A}:\ v'\equiv_x v$, i.e. $v'(y)=v(y)$ for all variables except for $x\}$
\item $\Vert(\exists x)\varphi\Vert_{\mathbf{M},v}^\mathcal{A}=\max\{\Vert\varphi\Vert_{M,v'}^\mathcal{A}:\ v'\equiv_x v$, i.e. $v'(y)=v(y)$ for all variables except for $x\}$.
\end{itemize}
\begin{remark}
Usually, the last two cases are defined by taking, respectively, $\inf$'s and $\sup$'s of truth values: since these $\inf$'s and $\sup$'s do not necessarily exist, we have to introduce the notion of safe model, if we drop the requirement that the model is finite. Conversely, every finite model is safe, and in particular it is also witnessed, in the sense of \cite{wit}: for this reason we can take $\min$ and $\max$.
\end{remark}
Let $\varphi(x_1,\dots,x_k)$ be a formula (i.e. a formula having $x_1,\dots,x_k$ as free variables), $\mathcal{A}$ be an MTL-chain, and $\mathbf{M}$ be a finite $\mathcal{A}$-model. With the notation $\Vert \varphi(a_1,\dots,a_k) \Vert^\mathcal{A}_{\mathbf{M}}$, with $a_1,\dots,a_k\in M$, we indicate $\Vert \varphi(x_1,\dots,x_k) \Vert^\mathcal{A}_{\mathbf{M},v}$, with $v(x_i)=a_i$.

\medskip
Let L be an axiomatic extension of MTL, and $\mathcal{A}$ be an L-chain. We say that L$\forall$ is complete w.r.t. the class of finite $\mathcal{A}$-models, whenever, for every (first-order) formula $\varphi$:
\begin{equation*}
\vdash_{\text{L}\forall}\varphi\qquad\text{iff}\qquad \Vert\varphi\Vert_{\mathbf{M},v}^\mathcal{A}=1,	
\end{equation*}
for every finite $\mathcal{A}$-model $\mathbf{M}$, and evaluation $v$.
\section{Incompleteness results}
In this section we present the first results generalizing Trakhtenbrot theorem. 

We begin by introducing a particular case of single chain completeness: this notion was initially studied in \cite{ssc1}.
\begin{definition}
Let L be an axiomatic extension of MTL. If there is an L-chain $\mathcal{A}$ such that L$\forall$ is complete w.r.t. the finite models of $\mathcal{A}$, then we say that L$\forall$ enjoys the finite single chain completeness (fSCC).
\end{definition}
Concerning the computational complexity, we define the following:
\begin{definition}
Let L be an axiomatic extension of MTL, and $\mathcal{A}$ be an L-chain. With TAUT$^\mathcal{A}$ we denote the set of tautologies associated to $\mathcal{A}$, and with  TAUT$_\text{L}$ we denote the set of tautologies associated to all L-chains. Moving to the first-order case, with fTAUT$^\mathcal{A}_\forall$ we denote the set of first-order tautologies associated to the finite models of $\mathcal{A}$.
\end{definition}
Clearly, if $\mathcal{A}$ is a generic chain for the variety of L-algebras, then TAUT$^\mathcal{A}=$TAUT$_\text{L}$.
In the next section we will show that if L is an extension of MTL s.t. TAUT$_\text{L}$ is decidable, then the fSCC fails to hold for L$\forall$.

This is a generalized version of the Trakhtenbrot theorem presented in \cite{trak1}, and subsequently extended to the many-valued case in \cite{trak}.

Moreover, in the next section, we will study more in detail the arithmetical complexity of the fTAUT$_\forall$ problem for many (first-order) axiomatic extensions of MTL, by showing that for a large family of logics it is $\Pi_1$-complete.

We start by recalling the classical Trakhtenbrot theorem: in the rest of the paper, with $\mathbf{2}$ we denote the two elements boolean algebra.
\begin{theorem}[\cite{trak1,vau,bgg}]\label{trak}
The set fTAUT$^\mathbf{2}_\forall$ is $\Pi_1$-complete.
\end{theorem}
Moving to the case of axiomatic extensions of MTL, we have that:
\begin{theorem}\label{main}
Let L be an axiomatic extension of MTL that is complete w.r.t. a chain $\mathcal{A}$, and such that TAUT$_\text{L}$ is decidable. Then fTAUT$^\mathcal{A}_\forall$ is $\Pi_1$. 
\end{theorem}
The proof is an adaptation of the one given in \cite{trak}.

The key point is a technique of coding formulas of predicate logic by some formulas of propositional logic. 
\begin{definition}\label{def:tr}
Let $\mathcal{A}$ be an MTL-chain, and $\mathbf{M}=\lag M, \{r_{P_i}\}_{P_i\in\mathbf{P}}\rog$ be a finite $\mathcal{A}$-model, with $\vert M\vert=n$.

For each predicate $P_i$ of arity $s$ we introduce $n^s$ propositional variables $p_{ij_1\dots j_s}$, where $j_1\dots j_s\in\{1,\dots,n\}$ (assume $M = \{1,\dots,n\}$). Define an $\mathcal{A}$-evaluation $e_M$ of these propositional variables by setting $e_M(p_{ij_1,\dots,j_s}) = r_{P_i}(j_1,\dots, j_s)$ (i.e. the truth value of $p_{ij_1,\dots,j_s}$ is the degree in which $(j_1,\dots,j_s)$ is in the relation $r_{P_i}$).
We work with formulas of predicate logic with free variables substituted by elements
of $M$. For each such object $\varphi$ we define its translation $\varphi^{*,n}$ as follows:
$(P_i(j_1,\dots,j_s))^{*,n}=p_{ij_1\dots j_s}$; $(\varphi\odot\psi)^{*,n}=\varphi^{*,n}\odot\psi^{*,n}$, for $\odot\in\{\&,\to,\land\}$; $(\bot)^{*,n}=\bot$; $((\forall x)\varphi(x))^{*,n}=\bigwedge_{i=1}^n\varphi^{*,n}(i)$; $((\exists x)\varphi(x))^{*,n}=\bigvee_{i=1}^n\varphi^{*,n}(i)$.

Note that if $\varphi$ is as assumed (free variables replaced by elements of $M$) then $\Vert\varphi\Vert^\mathcal{A}_{\mathbf{M}}$ has the meaning $\Vert\varphi\Vert^\mathcal{A}_{\mathbf{M},v}$ where $v$ just assigns to each free variable the corresponding element of $M$ (and otherwise arbitrary).
\end{definition}
\begin{lemma}\label{lem:tr}
For each MTL-chain $\mathcal{A}$, finite $\mathcal{A}$-model $\mathbf{M}$ of cardinality $n$, and $\varphi$ as above,
\begin{equation*}
\Vert\varphi\Vert^\mathcal{A}_{\mathbf{M}}=e_M(\varphi^{*,n}).
\end{equation*}
\end{lemma}
\begin{proof}
By structural induction over $\varphi$.
\begin{itemize}
\item If $\varphi$ is atomic or $\bot$, then the result follows immediately from \Cref{def:tr}.
\item Suppose that $\varphi$ has the form $\psi\odot\chi$ and that the claim holds for $\psi, \chi$, with $\odot\in\{\&,\to,\land\}$: let us call $\cdot$ the algebraic interpretation of $\odot$. By \Cref{def:tr} we have that $\Vert\varphi\Vert^\mathcal{A}_{\mathbf{M}}=\Vert\psi\Vert^\mathcal{A}_{\mathbf{M}}\cdot \Vert\chi\Vert^\mathcal{A}_{\mathbf{M}}= e_M(\psi^{*,n})\cdot e_M(\chi^{*,n})=e_M(\varphi^{*,n})$.

\item Consider the case in which $\varphi$ has the form $(\forall x)\psi(x)$, and the claim holds for $\psi$. W.l.o.g. assume $M = \{1,\dots,n\}$. We have that $\Vert\varphi\Vert^\mathcal{A}_{\mathbf{M}}=\min_{i=1}^n \Vert\psi(i)\Vert^\mathcal{A}_{\mathbf{M}}=\min_{i=1}^n e_M(\psi^{*,n}(i))=e_M(\bigwedge_{i=1}^n\psi^{*,n}(i))=e_M(\varphi^{*,n})$.

The case in which $\varphi$ has the form $(\exists x)\psi$ is almost identical (it is enough to replace $\min$ with $\max$, $\forall$ with $\exists$, and $\bigwedge$ with $\bigvee$), and hence the proof is complete.
\end{itemize}
\end{proof}
\begin{lemma}\label{clos}
Let L be an axiomatic extension of MTL: for every non-closed first-order formula $\varphi$ denote with $\varphi^c$ it universal closure. Then, for every generic L-chain $\mathcal{A}$, and non-closed formula $\varphi$:
\begin{equation*}
\varphi\in \text{fTAUT}^\mathcal{A}_{\forall}\qquad\text{iff}\qquad\varphi^c\in \text{fTAUT}^\mathcal{A}_{\forall}.
\end{equation*}
\end{lemma}
\begin{proof}
An easy check.
\end{proof}
We can now complete the proof of our first main result.	
\begin{proof}[Proof of \Cref{main}]
By \Cref{lem:tr} we have that, for every generic L-chain $\mathcal{A}$ and first-order closed (by \Cref{clos}, this can be done without loss of generality) formula $\varphi$: 
\begin{equation*}
\varphi\in \text{fTAUT}^\mathcal{A}_{\forall}\qquad\text{iff}\qquad(\forall n)(\varphi^{*,n}\in \text{TAUT}^\mathcal{A}).
\end{equation*}
Now, TAUT$^\mathcal{A}=$TAUT$_\text{L}$ is decidable, and every formula $\varphi^{*,n}$ can be computed in a finite time, since $\varphi$ contains only a finite number of predicates. Hence fTAUT$^\mathcal{A}_{\forall}$ is $\Pi_1$, and this concludes the proof.
\end{proof}
\section{Applications to many-valued logics}
In this section we analyze more in detail the arithmetical complexity of the fTAUT$_\forall$ problem, for a plethora of first-order many-valued logics, by showing that it is $\Pi_1$-complete. Moreover, we show that for every axiomatic extension L of MTL s.t. TAUT$_\text{L}$ is decidable, the fSCC fails to hold for L$\forall$.
\begin{theorem}\label{inc}
Let L be an axiomatic extension of MTL s.t. TAUT$_L$ is decidable, and whose corresponding variety is generated by an L-chain. Then, for every generic L-chain $\mathcal{A}$, if:
\begin{itemize}
\item L is a extension of BL or
\item L is an extension of WNM or
\item L is an extension of SMTL,
\end{itemize}
then fTAUT$^\mathcal{A}_{\forall}$ is $\Pi_1$-complete. 

In particular, if $\text{L}\in\{\text{DP},\text{G},\Pi\}$, then $fTAUT^\mathcal{A}_{\forall}$ is $\Pi_1$-complete for every infinite L-chain $\mathcal{A}$. 
\end{theorem}
To prove the theorem, we first need to develop some machinery for the case of the extensions of WNM, in the way to recursively reduce the problem to the G\"odel-case.  

We start by adapting a translation of formulas firstly presented in \cite{nmfirst} for NM$\forall$. 
\begin{definition}\label{def:trg}
Let $\varphi$ be a first-order formula. We define $\varphi^*$, inductively, as follows:
\begin{itemize}
\item If $\varphi$ is atomic, then $\varphi^*\df\varphi^2$.
\item If $\varphi$ is $\bot$, then $\varphi^*\df\bot$.
\item If $\varphi$ is $\psi\land\chi$, then $\varphi^*\df\psi^*\land\chi^*$.
\item If $\varphi$ is $\psi\&\chi$, then $\varphi^*\df\psi^*\&\chi^*$.
\item If $\varphi$ is $\psi\rightarrow\chi$, then $\varphi^*\df(\psi^*\rightarrow\chi^*)^2$.
\item If $\varphi$ is $(\forall x)\chi$, then $\varphi^*\df(\forall x)\chi^*$.
\item If $\varphi$ is $(\exists x)\chi$, then $\varphi^*\df(\exists x)\chi^*$.
\end{itemize}
\end{definition}
\begin{definition}\label{def:cut}
Let $\mathcal{A}$ be a WNM-chain, and $\mathbf{M}=\lag M,\{r_P\}_{P\in\mathbf{P}}\rog$ be a finite $\mathcal{A}$-model. We construct the $\mathcal{A}$-model  $\mathbf{M^+}=\lag M,\{r'_P\}_{P\in\mathbf{P}}\rog$ as follows: for every predicate $P$ of arity $n$, and $m_1,\dots,m_n\in M$,
\begin{equation*}
r'_P(m_1,\dots,m_n)\df\begin{cases}
r_P(m_1,\dots,m_n)&\text{if }r_P(m_1,\dots,m_n)\in A^+,\\
0&\text{otherwise.}
\end{cases}
\end{equation*}
\end{definition}
The idea is that $\mathbf{M}^+$ restricts the assignments of $\mathbf{M}$ (to the atomic formulas) to the idempotent elements of $\mathcal{A}$.

\medskip
\noindent Observe now that:
\begin{lemma}\label{lemg}
For every WNM-chain $\mathcal{A}$, and every $x,y\in A^+$:
\begin{equation*}
x*y=\min\{x,y\}\qquad x\Rightarrow y=\begin{cases}
1&\text{if }x\leq y\\
y&\text{otherwise.}
\end{cases}
\end{equation*}
That is, $*,\Rightarrow$ are the operations of a G\"odel hoop (see \cite{eghm} for details).
\end{lemma}
\begin{proof}
An easy check from \Cref{eq:op}.
\end{proof}
The next two lemmas show the connection between G\"odel chains and WNM-chains with restricted models.
\begin{lemma}\label{lem:gc}
Let $\mathcal{A}$ be a WNM-chain, and $\varphi$ be a first-order formula. Then for every finite $\mathcal{A}$-model $\mathbf{M}$, and evaluation $v$ it holds that:
\begin{equation*}
\Vert \varphi^*\Vert^\mathcal{A}_{\mathbf{M},v}=\Vert \varphi^*\Vert_{\mathbf{M}^+,v}^\mathcal{A}.
\end{equation*}
Moreover, $\Vert \varphi^*\Vert_{\mathbf{M}^+,v}^\mathcal{A}\in A^+\cup\{0\}$.
\end{lemma}
\begin{proof}
By structural induction over $\varphi$.
\begin{itemize}
\item If $\varphi$ is atomic or $\bot$, then the result follows immediately from \Cref{def:cut} and \Cref{def:trg}.
\item Suppose that $\varphi$ has the form $\psi\land\chi$ and that the claim holds for $\psi, \chi$. 

By the induction hypothesis we have that $\Vert \varphi^*\Vert^\mathcal{A}_{\mathbf{M},v}=\min\{\Vert \psi^*\Vert^\mathcal{A}_{\mathbf{M},v},\Vert \chi^*\Vert^\mathcal{A}_{\mathbf{M},v}\}=\min\{\Vert \psi^*\Vert^\mathcal{A}_{\mathbf{M}^+,v},\Vert \chi^*\Vert^\mathcal{A}_{\mathbf{M}^+,v}\}=\Vert \varphi^*\Vert^\mathcal{A}_{\mathbf{M}^+,v}$. Since $\Vert \psi^*\Vert^\mathcal{A}_{\mathbf{M}^+,v},\Vert \chi^*\Vert^\mathcal{A}_{\mathbf{M}^+,v}\in A^+\cup\{0\}$, then also $\Vert \varphi^*\Vert^\mathcal{A}_{\mathbf{M}^+,v}\in A^+\cup\{0\}$.
\item Suppose that $\varphi$ has the form $\psi\&\chi$, and that the claim holds for $\psi, \chi$. 

By the induction hypothesis we have that $\Vert \psi^*\Vert^\mathcal{A}_{\mathbf{M}^+,v},\Vert \chi^*\Vert^\mathcal{A}_{\mathbf{M}^+,v}\in A^+\cup\{0\}$, and $\Vert \psi^*\Vert^\mathcal{A}_{\mathbf{M},v}=\Vert \psi^*\Vert^\mathcal{A}_{\mathbf{M}^+,v},\Vert \chi^*\Vert^\mathcal{A}_{\mathbf{M},v}=\Vert \chi^*\Vert^\mathcal{A}_{\mathbf{M}^+,v}$. From this fact, and \Cref{lemg} we have that $\Vert \varphi^*\Vert^\mathcal{A}_{\mathbf{M},v}=\min\{\Vert \psi^*\Vert^\mathcal{A}_{\mathbf{M},v},\Vert \chi^*\Vert^\mathcal{A}_{\mathbf{M},v}\}$, and the proof is identical to the previous case.
\item Suppose that $\varphi$ has the form $\psi\to\chi$, and that the claim holds for $\psi, \chi$. 

By the induction hypothesis we have that $\Vert \psi^*\Vert^\mathcal{A}_{\mathbf{M}^+,v},\Vert \chi^*\Vert^\mathcal{A}_{\mathbf{M}^+,v}\in A^+\cup\{0\}$, and $\Vert \psi^*\Vert^\mathcal{A}_{\mathbf{M},v}=\Vert \psi^*\Vert^\mathcal{A}_{\mathbf{M}^+,v},\Vert \chi^*\Vert^\mathcal{A}_{\mathbf{M},v}=\Vert \chi^*\Vert^\mathcal{A}_{\mathbf{M}^+,v}$. If $\Vert \psi^*\Vert^\mathcal{A}_{\mathbf{M},v}\leq\Vert \chi^*\Vert^\mathcal{A}_{\mathbf{M},v}$, then $\Vert \varphi^*\Vert^\mathcal{A}_{\mathbf{M},v}=\Vert \varphi^*\Vert^\mathcal{A}_{\mathbf{M}^+,v}=1$. Assume now that $\Vert \psi^*\Vert^\mathcal{A}_{\mathbf{M},v}>\Vert \chi^*\Vert^\mathcal{A}_{\mathbf{M},v}$: if $\Vert \chi^*\Vert^\mathcal{A}_{\mathbf{M},v}=0$, then $\Vert \varphi^*\Vert^\mathcal{A}_{\mathbf{M},v}=(\Vert \psi^*\Vert^\mathcal{A}_{\mathbf{M},v}\Rightarrow 0)^2=(\Vert \psi^*\Vert^\mathcal{A}_{\mathbf{M^+},v}\Rightarrow 0)^2=\Vert \varphi^*\Vert^\mathcal{A}_{\mathbf{M^+},v}$. Since $\mathcal{A}$ is a WNM-chain, and  $\Vert \psi^*\Vert^\mathcal{A}_{\mathbf{M},v}\in A^+$, then $\Vert \psi^*\Vert^\mathcal{A}_{\mathbf{M},v}\Rightarrow 0\in A\setminus A^+$, and hence $(\Vert \psi^*\Vert^\mathcal{A}_{\mathbf{M},v}\Rightarrow 0)^2=0$. It follows that $\Vert \varphi^*\Vert^\mathcal{A}_{\mathbf{M},v}=\Vert \varphi^*\Vert^\mathcal{A}_{\mathbf{M^+},v}=0$.
Finally, suppose that $\Vert \psi^*\Vert^\mathcal{A}_{\mathbf{M},v}>\Vert \chi^*\Vert^\mathcal{A}_{\mathbf{M},v}>0$: it follows that $\Vert \psi^*\Vert^\mathcal{A}_{\mathbf{M},v}, \Vert \chi^*\Vert^\mathcal{A}_{\mathbf{M},v}\in A^+$, and by \Cref{lemg} $\Vert \varphi^*\Vert^\mathcal{A}_{\mathbf{M},v}=(\Vert \psi^*\Vert^\mathcal{A}_{\mathbf{M},v}\Rightarrow \Vert \chi^*\Vert^\mathcal{A}_{\mathbf{M},v})^2=\Vert \chi^*\Vert^\mathcal{A}_{\mathbf{M},v}=\Vert \chi^*\Vert^\mathcal{A}_{\mathbf{M^+},v}=(\Vert \psi^*\Vert^\mathcal{A}_{\mathbf{M^+},v}\Rightarrow \Vert \chi^*\Vert^\mathcal{A}_{\mathbf{M^+},v})^2=\Vert \varphi^*\Vert^\mathcal{A}_{\mathbf{M^+},v}$.
\item Consider the case in which $\varphi$ has the form $(\forall x)\psi$, and the claim holds for $\psi$. Since we are working on a finite model, then there is a particular evaluation $w$ such that $\Vert \psi^*\Vert^\mathcal{A}_{\mathbf{M},w}=\min_{u\equiv_x v}\{\Vert \psi^*\Vert^\mathcal{A}_{\mathbf{M},u}\}=\Vert (\forall x)\psi^*\Vert^\mathcal{A}_{\mathbf{M},v}$. Since, by the induction hypothesis, $\Vert \psi^*\Vert^\mathcal{A}_{\mathbf{M},t}=\Vert \psi^*\Vert^\mathcal{A}_{\mathbf{M^+},t}$ for every $t$ (and each of them belongs to $A^+\cup\{0\}$), it follows that $\Vert (\forall x)\psi^*\Vert^\mathcal{A}_{\mathbf{M},v}=\Vert (\forall x)\psi^*\Vert^\mathcal{A}_{\mathbf{M^+},v}$. Clearly, $\Vert (\forall x)\psi^*\Vert^\mathcal{A}_{\mathbf{M^+},v}\in A^+\cup\{0\}$. 

The case in which $\varphi$ has the form $(\exists x)\psi$ is almost identical (it is enough to replace $\min$ with $\max$, and $\forall$ with $\exists$), and hence the proof is complete.
\end{itemize}
\end{proof}

\begin{lemma}\label{lem:gc1}
Let $\mathcal{A}$ be a WNM-chain, and $\varphi$ be a first-order formula. Then for every finite $\mathcal{A}$-model $\mathbf{M}$, and evaluation $v$ it holds that:
\begin{equation*}
\Vert \varphi^*\Vert^\mathcal{A}_{\mathbf{M^+},v}=\Vert \varphi^*\Vert^{\mathcal{A}_G}_{\mathbf{M}',v}=\Vert \varphi\Vert^{\mathcal{A}_G}_{\mathbf{M}',v}.
\end{equation*}
Where $\mathcal{A}_G$ is the G\"odel chain with support $A^+\cup\{0\}$, and $\mathbf{M}'$ is identical to $\mathbf{M}^+$ with the only difference that the codomain of the various $r_P$'s is $A^+\cup\{0\}$.
\end{lemma}
\begin{proof}
The first equality can be shown by structural induction over $\varphi$, by inspecting the proof of \Cref{lem:gc}, and considering \Cref{lemg}. 

The second equality of the theorem is immediate from \Cref{def:trg}, and the fact that the equation $x=x^2$ holds in every G\"odel chain.
\end{proof}
We also need the following results, concerning \L ukasiewicz logics: they are adaptations of \cite[Lemma 4.15, Theorem 4.16]{lcomp}.
\begin{definition}
For every atomic formula $P(\vv{x})$, define PREDEF$_P\df (\forall \vv{x})\neg(P(\vv{x})\leftrightarrow \neg P(\vv{x}))$. For every formula $\varphi$, with PREDEF$_\varphi$ we denote the $\land$ conjunction of PREDEF$_P$, for every atomic formula $P(\vv{x})$ in $\varphi$.

\noindent Let us call \emph{classical} a formula containing only $\land,\vee,\neg$ as connectives, and $\forall$ as quantifier.
\end{definition}
\begin{lemma}\label{pred}
\begin{enumerate}
\item For every classical formula $\varphi$, every MV-chain $\mathcal{A}$, and finite $\mathcal{A}$-model $\mathbf{M}$, if $\Vert\text{PREDEF}_\varphi\Vert_{\mathbf{M},w}^\mathcal{A}>0$, for some evaluation $w$, then $\Vert\text{PREDEF}_\varphi\Vert_{\mathbf{M},v}^\mathcal{A}>0$ for every other evaluation $v$.

\item Moreover, if $\Vert\text{PREDEF}_\varphi\Vert_{\mathbf{M},v}^\mathcal{A}>0$ for some evaluation $v$, then $\Vert\psi\Vert_{\mathbf{M},v}^\mathcal{A}$ is not a negation fixpoint, for every subformula $\psi$ of $\varphi$.
\end{enumerate}
\end{lemma}
\begin{proof}
\begin{enumerate}
\item Immediate, because PREDEF$_\varphi$ is a $\land$ conjunction of universally quantified closed formulas, and $\mathbf{M}$ is a finite model.
\item By structural induction over $\varphi$ (recall that $\varphi$ is classical): the cases in which $\varphi$ is atomic or has the form $\neg\psi$, $\psi\land\chi$, $\psi\vee\chi$ are immediate. We now analyze the case in which $\varphi\df(\forall x)\psi$, and that the claim holds for $\psi$. By the first part of the lemma and the hypothesis, we have that $\Vert\text{PREDEF}_\varphi\Vert_{\mathbf{M},w}^\mathcal{A}>0$, for every evaluation $w$: hence by the induction hypothesis we have that $\Vert\psi\Vert^\mathcal{A}_{\mathbf{M},w}$ is not a negation fixpoint, for every evaluation $w$. Since $\mathbf{M}$ is finite, it follows that $\Vert(\forall x)\psi\Vert^\mathcal{A}_{\mathbf{M},v}$ is not a negation fixpoint, and this concludes the proof.
\end{enumerate}

\end{proof}
\begin{lemma}\label{lem:luk1}
Let $\mathcal{A}$ be an MV-chain. For every finite $\mathcal{A}$-model $\mathbf{M}$, construct a finite $\mathbf{2}$-model $\mathbf{M}'$, with $M'=M$, and such that, for every atomic formula $\psi$, and valuation $v$, $\Vert\psi\Vert^\mathbf{2}_{\mathbf{M}',v}=1$ if $\Vert\psi\Vert^\mathcal{A}_{\mathbf{M},v}\in A^+$ and $\Vert\psi\Vert^\mathbf{2}_{\mathbf{M}',v}=0$ otherwise. 

We have that, for every classical formula $\varphi$, if $\Vert\text{PREDEF}_\varphi\Vert_{\mathbf{M},v}^\mathcal{A}>0$, for some evaluation $v$, then it holds that: 
\begin{equation*}
\Vert\varphi\Vert^\mathcal{A}_{\mathbf{M},v}\in A^+\qquad\text{iff}\qquad\Vert\varphi\Vert^\mathbf{2}_{\mathbf{M}',v}=1.
\end{equation*}
\end{lemma}
\begin{proof}
By structural induction over $\varphi$, assuming that $\Vert\text{PREDEF}_\varphi\Vert_{\mathbf{M},v}^\mathcal{A}>0$, for some evaluation $v$.
\begin{itemize}
\item If $\varphi$ is atomic, or has the form $\psi\land\chi$ or $\psi\vee\chi$, then the claim is immediate by the definition of $\mathbf{M}'$, and the induction hypothesis.
\item Suppose that $\varphi\df\neg\psi$, and that the claim holds for $\psi$. If $\Vert\varphi\Vert^\mathcal{A}_{\mathbf{M},v}\in A^+$, then $\Vert\psi\Vert^\mathcal{A}_{\mathbf{M},v}\notin A^+$, and this implies that $\Vert\psi\Vert^\mathbf{2}_{\mathbf{M}',v}=0$ and $\Vert\neg\psi\Vert^\mathbf{2}_{\mathbf{M}',v}=1$. If $\Vert\varphi\Vert^\mathcal{A}_{\mathbf{M},v}\notin A^+$, then $\Vert\psi\Vert^\mathcal{A}_{\mathbf{M},v}\in A^+$: indeed, if not, then $\Vert\psi\Vert^\mathcal{A}_{\mathbf{M},v}$ would be a negation fixpoint, in contrast with \Cref{pred}. Hence $\Vert\psi\Vert^\mathcal{A}_{\mathbf{M},v}\in A^+$, that implies $\Vert\psi\Vert^\mathbf{2}_{\mathbf{M}',v}=1$ and $\Vert\neg\psi\Vert^\mathbf{2}_{\mathbf{M}',v}=0$.
\item Finally, suppose that $\varphi\df(\forall x)\psi$, and that the claim holds for $\psi$. We have that $\Vert(\forall x)\psi\Vert^\mathcal{A}_{\mathbf{M},v}\in A^+$ iff for every evaluation $w$, $\Vert\psi\Vert^\mathcal{A}_{\mathbf{M},w}\in A^+$, that happens iff (by the induction hypothesis and \Cref{pred}) $\Vert\psi\Vert^\mathbf{2}_{\mathbf{M}',w}=1$ for every evaluation $w$, that happens iff $\Vert(\forall x)\psi\Vert^\mathbf{2}_{\mathbf{M}',v}=1$.
\end{itemize}
\end{proof}
\begin{lemma}\label{lem:luk}
Let $\mathcal{A}$ be an MV-chain, and $\varphi$ be a first-order classical formula. Let $\varphi^*$ be $\neg\text{PREDEF}_\varphi\vee (\neg\varphi\to\varphi)$. Then,
\begin{equation*}
\varphi^*\in \text{fTAUT}^\mathcal{A}_\forall\qquad\text{iff}\qquad \varphi\in \text{fTAUT}^{\mathbf{2}}_\forall.
\end{equation*}
Hence fTAUT$^\mathcal{A}_\forall$ is $\Pi_1$-hard.
\end{lemma}
\begin{proof}

Suppose that $\varphi^*\in\text{fTAUT}_\forall^\mathcal{A}$: then $\varphi^*\in\text{fTAUT}_\forall^\mathbf{2}$, being $\mathbf{2}$ a subalgebra of $\mathcal{A}$. In particular, $\text{PREDEF}_\varphi\in\text{fTAUT}_\forall^\mathbf{2}$, and hence we must have that $\neg\varphi\to\varphi\in\text{fTAUT}_\forall^\mathbf{2}$. Now, since $\vdash_{\text{BOOL}\forall} (\neg\varphi\to\varphi)\to\varphi$, it follows that $\varphi\in\text{fTAUT}_\forall^\mathbf{2}$.

Conversely, suppose that $\varphi^*\notin\text{fTAUT}_\forall^\mathcal{A}$: then there is a finite $\mathcal{A}$-model $\mathbf{M}$, and an evaluation $v$ such that $\Vert\varphi^*\Vert^\mathcal{A}_{\mathbf{M},v}<1$. It follows that $\Vert\text{PREDEF}_\varphi\Vert^\mathcal{A}_{\mathbf{M},v}>0$, and $\Vert\varphi\Vert^\mathcal{A}_{\mathbf{M},v}<\Vert\neg\varphi\Vert^\mathcal{A}_{\mathbf{M},v}$: i.e. $\Vert\varphi\Vert^\mathcal{A}_{\mathbf{M},v}\notin A^+$. By \Cref{lem:luk1} we can construct a finite $\mathbf{2}$-model $\mathbf{M}'$ such that $\Vert\varphi\Vert^\mathbf{2}_{\mathbf{M}',v}=0$; hence $\varphi\notin\text{fTAUT}_\forall^\mathbf{2}$. 

By \Cref{trak}, $\text{fTAUT}^{\mathbf{2}}_\forall$ is $\Pi_1$-complete, and hence  $\text{fTAUT}^\mathcal{A}_\forall$ is $\Pi_1$-hard.
\end{proof}
We can finally return to the proof of our main theorem.
\begin{proof}[Proof of \Cref{inc}]
Let L be an axiomatic extension of MTL s.t. TAUT$_\text{L}$ is decidable, and whose corresponding variety is generated by a chain. By \Cref{main}, for every generic L-chain $\mathcal{A}$, fTAUT$^\mathcal{A}_\forall$ is $\Pi_1$.

Note: in the rest of the proof we assume to work with \emph{first-order} formulas.

Let L be an axiomatic extension of BL s.t. TAUT$_\text{L}$ is decidable and whose corresponding variety is generated by a chain. We show that, for every generic L-chain $\mathcal{A}$, fTAUT$^\mathcal{A}_\forall$ is $\Pi_1$-complete: it remains only to prove the hardness.

By \cite[Theorem 3.7]{am} we have that every BL-chain is isomorphic to an ordinal sum of Wajsberg hoops (see \cite{am,eghm} for details about ordinal sums and hoops) whose first component is an MV-chain. Let $\mathcal{A}$ be a generic L-chain, and let $\mathcal{B}$ be its first component, in the decomposition as ordinal sum\footnote{Clearly, if L is an extension of {\L}, then $\mathcal{A}$ and $\mathcal{B}$ coincide.}.

For every formula $\varphi$ let $t(\varphi)$ the formula obtained from $\varphi$ by replacing its atomic formulas with their double negations. Take $x\in A$: by the definition of ordinal sum, if $x$ belongs to the first component, then $\sim\sim x=x$, otherwise $\sim\sim x=1$. As a consequence for every formula $\varphi$:
\begin{equation*}
t(\varphi)\in \text{fTAUT}^\mathcal{A}_{\forall}\qquad\text{iff}\qquad \varphi\in \text{fTAUT}^{\mathcal{B}}_{\forall}.
\end{equation*}
By \Cref{main} and \Cref{lem:luk} we have that $\text{fTAUT}^{\mathcal{B}}_{\forall}$ is $\Pi_1$-complete. Since the translation $t$ is computable in a low complexity time, it follows that fTAUT$^\mathcal{A}_{\forall}$ is $\Pi_1$-complete. 

In \cite{dum,prodcig} it is shown that a G\"odel chain (product chain) is generic for the variety of G\"odel (product) algebras if and only if it is infinite, and hence for $\text{L}\in \text{G},\Pi$ the set fTAUT$^\mathcal{A}_{\forall}$ is $\Pi_1$-complete for every infinite L-chain $\mathcal{A}$.

Let now L be an axiomatic extension of SMTL s.t. TAUT$_\text{L}$ is decidable and whose corresponding variety is generated by a chain. For every generic L-chain $\mathcal{A}$, we show the $\Pi_1$-completeness of fTAUT$^\mathcal{A}_{\forall}$ by recursively reducing it to the classical case. 

For every formula $\varphi$ let $t(\varphi)$ be the formula obtained from $\varphi$ by replacing its atomic formulas with their double negations. Since for every L-chain and element $x$ it holds that $\sim\sim x=1$ if $x>0$, and $\sim\sim x=0$ if $x=0$, then we have that, for every generic L-chain $\mathcal{A}$, and formula $\varphi$:
\begin{equation*}
t(\varphi)\in \text{fTAUT}^\mathcal{A}_{\forall}\qquad\text{iff}\qquad \varphi\in \text{fTAUT}^{\mathbf{2}}_{\forall}.
\end{equation*}
Since the translation $t$ is computable in a low complexity time, it follows that fTAUT$^\mathcal{A}_{\forall}$ is $\Pi_1$-complete. 

Suppose that L is an extension of WNM s.t. TAUT$_\text{L}$ is decidable, and whose corresponding variety is generated by a chain. By \Cref{lem:gc}, and \Cref{lem:gc1} we have that for every generic L-chain $\mathcal{A}$, and formula $\varphi$:
\begin{equation*}
\varphi^*\in \text{fTAUT}^\mathcal{A}_{\forall}\qquad\text{iff}\qquad \varphi\in \text{fTAUT}^{\mathcal{A}_G}_{\forall}.
\end{equation*}
By the first part of the proof we know that fTAUT$^{\mathcal{A}_G}_{\forall}$ is $\Pi_1$-complete: moreover, as shown in \cite{dm}, every G\"odel chain generates $\mathbb{G}$ (if it is infinite), or $\mathbb{G}_k$, for some $k$ (if it is finite and has $k$ elements). For $k=2$, clearly $\mathbb{G}_k$ is the variety of boolean algebras. It follows that fTAUT$^\mathcal{A}_{\forall}$ is $\Pi_1$-complete. 

Finally, if L is DP (the fact that TAUT$_\text{DP}$ is decidable easily follows from the results of \cite{abv}), then as pointed out in \cite{abv}, a DP-chain is generic if and only if it is infinite. It follows that fTAUT$^\mathcal{A}_{\forall}$ is $\Pi_1$-complete, for every infinite DP-chain.
\end{proof}
\noindent As a consequence of \Cref{inc} we have that:
\begin{corollary}
Let L be one of BL, BL$_n$, {\L}, {\L}$_n$, G, G$_n$, $\Pi$, SMTL, SBL, SBL$^n$, SBL$_n$, WNM, NM, NMG, RDP, DP, and $\mathcal{A}$ be a generic L-chain. Then the set fTAUT$^\mathcal{A}_{L\forall}$ is $\Pi_1$-complete.
\end{corollary}
\begin{proof}
Let L be one of these logics. The fact that all the corresponding varieties are generated by an L-chain is shown in \cite{blvar,bln,chang,grig,dum,dm,prod,ssc1,nog,eg,nmg,rdp,abv}. Moreover, in \cite{bhmv,bln,munnp,haj,elm,agh,nogwnm,nmgc,bval} it is shown (or it is easy to check) that for all these logics TAUT$_\text{L}$ is decidable. Hence, by \Cref{inc}, we have that for every generic L-chain $\mathcal{A}$, fTAUT$^\mathcal{A}_{\forall}$ is $\Pi_1$-complete.
\end{proof}
\begin{remark}
In a personal communication, Félix Bou pointed out that some of the results of \Cref{inc} (in particular the ones concerning \L ukasiewicz logic and BL) can also be proved by using some of the results (still unpublished) that he presented to a conference in 2012 (see \cite{bou} for the presentation). 
\end{remark}
\begin{problem}
Let L be an axiomatic extension of IMTL (i.e. MTL plus \ref{inv}) whose corresponding variety is generated by a chain: given a generic L-chain $\mathcal{A}$, in which cases fTAUT$^\mathcal{A}_{\forall}$ is $\Pi_1$-complete?
\end{problem}
We conclude with a negative result, concerning the fSCC.
\begin{theorem}\label{noc}
Let L be an axiomatic extension of MTL s.t. TAUT$_\text{L}$ is decidable. Then the fSCC fails to hold, for L$\forall$. 
\end{theorem}
\begin{proof}
Let L be a such logic. Note that every L-chain $\mathcal{A}$ that is not complete w.r.t. L cannot be complete w.r.t. L$\forall$, even if we restrict to finite models. Indeed, for a such chain there is a propositional formula $\varphi(x_1,\dots,x_k)$ such that $\mathcal{A}\models \varphi(x_1,\dots,x_k)$ and $\not\vdash_\text{L}\varphi(x_1,\dots,x_k)$. Hence there is an L-chain $\mathcal{B}$ and an evaluation $v$ such that $v(\varphi(x_1,\dots,x_k))<1$. Let $\psi$ be the first order formula obtained from $\varphi(x_1,\dots,x_k)$ by replacing every variable $x_i$ with a unary predicate $P_i(x_i)$. Take now a $\mathcal{B}$-model $\mathbf{M}$ and an evaluation $w$ such that $M=\{c\}$, and $\Vert P_i(x_i)\Vert_{\mathbf{M},w}^\mathcal{B}=v(x_i)$: clearly $\mathcal{B}\not\models\psi$, and hence $\text{L}\forall\not\vdash\psi$. Since $\mathcal{A}\models\varphi(x_1,\dots,x_k)$ it easy to check that $\Vert \psi\Vert_{\mathbf{M}',w'}^\mathcal{A}$, for every finite $\mathcal{A}$-model $\mathbf{M}'$, and every evaluation $w'$: hence $\psi\in\text{fTAUT}^\mathcal{A}_\forall$. 

So, for every extension L of MTL without a generic chain, the fSCC fails to hold, for L$\forall$.

Conversely, let now $\mathcal{A}$ be a generic L-chain: by \Cref{main},  fTAUT$^\mathcal{A}_{\forall}$ is $\Pi_1$. In \cite{mono} it is shown that every logic between MTL$\forall$ and BOOL$\forall$ is undecidable, and the proof applies also to the case of languages containing only predicates (with a least a binary one). Since TAUT$_{\text{L}\forall}$ is $\Sigma_1$ (being recursively enumerable), it cannot be also $\Pi_1$, otherwise it would be decidable, a contradiction. Hence L$\forall$ cannot be complete w.r.t. the finite models of $\mathcal{A}$.

It follows that, given an axiomatic extension L of MTL s.t. TAUT$_\text{L}$ is decidable, the fSCC fails to hold, for L$\forall$.
\end{proof}
\subsection{Axiomatic extensions of MTL with Baaz operator $\Delta$}
We conclude the paper by analyzing the axiomatic extensions of MTL expanded with the Baaz operator $\Delta$, firstly introduced in \cite{bdelta} (see \cite{hand,dist} for other details). For every axiomatic extension L of MTL, we denote with L$_\Delta$ its expansion with an operator $\Delta$ satisfying the following axioms,
\begin{align*}
\tag*{($\Delta$1)}&\Delta(\varphi) \vee \neg\Delta(\varphi).\\
\tag*{($\Delta$2)}&\Delta(\varphi \vee \psi)\rightarrow ((\Delta(\varphi) \vee \Delta(\psi))).\\
\tag*{($\Delta$3)}&\Delta(\varphi)\rightarrow \varphi.\\
\tag*{($\Delta$4)}&\Delta(\varphi)\rightarrow \Delta(\Delta(\varphi)).\\
\tag*{($\Delta$5)}&\Delta(\varphi\rightarrow \psi)\rightarrow (\Delta(\varphi)\rightarrow \Delta(\psi)).
\end{align*}
and the following additional inference rule: $\frac{\varphi}{\Delta\varphi}$.

We recall that on every MTL$_\Delta$-chain $\mathcal{A}$, if we call $\delta$ the algebraic corresponding to $\Delta$ connective, then for every $x \in A$, it holds that $\delta(x) = 1$ if $x = 1$, whilst $\delta(x) = 0$ if $x < 1$. Given an MTL-chain $\mathcal{A}$, with $\mathcal{A}_\Delta$ we denote its expansion with the $\delta$ operation.

\begin{theorem}\label{delta}
Let L be an axiomatic extension of MTL$_\Delta$ whose corresponding variety is generated by an L-chain. If TAUT$_\text{L}$ is decidable, then for every generic L-chain $\mathcal{A}$ it holds that fTAUT$^\mathcal{A}_{\forall}$ is $\Pi_1$-complete.
\end{theorem}
\begin{proof}
We first need to modify \Cref{def:tr}, by adding the case $(\Delta\varphi)^{*,n}=\Delta(\varphi^{*,n})$: a direct inspection shows that \Cref{lem:tr} works also in this case. Again, with a proof almost identical to the one of \Cref{main}, we can show that fTAUT$^\mathcal{A}_{\text{L}\forall}$ is $\Pi_1$ (note that the requirement that TAUT$_\text{L}$ is decidable is essential). For every formula $\varphi$ let $t(\varphi)$ the formula obtained from $\varphi$ by replacing every atomic formula $A$ with $\Delta A$. Since for every L-chain and element $x$ it holds that $\delta(x)=0$ if $x<1$, and $\delta(x)=1$ if $x=1$, then we have that, for every generic L-chain $\mathcal{A}$, and formula $\varphi$:
\begin{equation*}
t(\varphi)\in \text{fTAUT}^\mathcal{A}_{\forall}\qquad\text{iff}\qquad \varphi\in \text{fTAUT}^{\mathbf{2}_\Delta}_{\forall}.
\end{equation*}
Now, it is easy to check that over classical logic $\varphi\leftrightarrow\Delta\varphi$ is a tautology: indeed, over $\mathbf{2}$, $\delta$ behaves like the identity map. In other terms, the addition of $\Delta$ does not change the expressive power of the language, in classical logic (propositional or first-order). Hence $\text{fTAUT}^{\mathbf{2}_\Delta}_{\forall}$ is $\Pi_1$-complete.
Since the translation $t$ is computable in a low complexity time, it follows that fTAUT$^\mathcal{A}_{\forall}$ is $\Pi_1$-complete.
\end{proof}
\begin{remark}
\begin{itemize}
\item One can ask if, given an axiomatic extension L of MTL having a generic L-chain $\mathcal{A}$, the chain $\mathcal{A}_\Delta$ is generic for $\mathbb{L}_\Delta$ or not. The answer is negative, in general: consider the subalgebra of standard MV-algebra with support given by all the rational between $0$ and $1$ with odd denominator, and call it $\mathcal{A}$. By \cite[Proposition 8.1.1]{mun} we have that $\mathcal{A}$ generates the variety of MV-algebras: however, $\mathcal{A}_\Delta$ is not generic for the variety of $\text{\L}_\Delta$-algebras. Indeed, it is easy to check that the formula 
\begin{equation}\label{fdelta}
\tag*{(f)}\Delta(\varphi\leftrightarrow\neg\varphi)\to\varphi
\end{equation}
holds in an MTL$_\Delta$-chain if and only if it does not have a negation fixpoint. Hence $\mathcal{A}_\Delta\models\text{\ref{fdelta}}$, whilst $[0,1]_{\text{\L}_\Delta}\not\models\text{\ref{fdelta}}$: hence $\text{\L}_\Delta\nvdash\text{\ref{fdelta}}$, and $\mathcal{A}_\Delta$ cannot be generic for the variety of \L$_\Delta$-algebras.

This counterexample shows how, for some logics, the addition of the $\Delta$ operator is non-trivial, in terms of the expressive power of the resulting logic (this is not the case for the classical one, as already explained). Indeed, with $\Delta$ we can construct formulas that capture algebraic properties that cannot be described without it.
\item More in general, one can ask if, given an axiomatic extension L of MTL enjoying the single chain completeness, its expansion L$_\Delta$ enjoys the SCC or not. Actually, it is an open problem and, in the light of the previous counterexample, it is also non-trivial.
\end{itemize}
\end{remark}
We conclude with a result analogous to \Cref{noc}.
\begin{theorem}\label{noc1}
Let L be an axiomatic extension of MTL$_\Delta$ such that $TAUT_{L}$ is decidable. Then the fSCC fails to holds, for L$\forall$.
\end{theorem} 
\begin{proof}
Let L be a such logic: since it is a $\Delta$-core fuzzy logic, in the sense of \cite{dist}, then both L and L$\forall$ are complete w.r.t. the class of all L-chains (and all models, for L$\forall$). 

Note that every L-chain $\mathcal{A}$ that is not complete w.r.t. L cannot be complete w.r.t. L$\forall$, even if we restrict to finite models. Indeed, for a such chain there is a propositional formula $\varphi(x_1,\dots,x_k)$ such that $\mathcal{A}\models \varphi(x_1,\dots,x_k)$ and $\not\vdash_\text{L}\varphi(x_1,\dots,x_k)$. Hence there is an L-chain $\mathcal{B}$ and an evaluation $v$ such that $v(\varphi(x_1,\dots,x_k))<1$. Let $\psi$ be the first order formula obtained from $\varphi(x_1,\dots,x_k)$ by replacing every variable $x_i$ with a unary predicate $P_i(x_i)$. Take now a $\mathcal{B}$-model $\mathbf{M}$ and an evaluation $w$ such that $M=\{c\}$, and $\Vert P_i(x_i)\Vert_{\mathbf{M},w}^\mathcal{B}=v(x_i)$: clearly $\mathcal{B}\not\models\psi$, and hence $\text{L}\forall\not\vdash\psi$. Since $\mathcal{A}\models\varphi(x_1,\dots,x_k)$ it easy to check that $\Vert \psi\Vert_{\mathbf{M}',w'}^\mathcal{A}$, for every finite $\mathcal{A}$-model $\mathbf{M}'$, and every evaluation $w'$: hence $\psi\in\text{fTAUT}^\mathcal{A}_\forall$.

Conversely, if an L-chain $\mathcal{A}$ is complete w.r.t. L, then by \Cref{delta} we have that fTAUT$^\mathcal{A}_{\forall}$ is $\Pi_1$-complete. Since the set of all theorems of L$\forall$ is $\Sigma_1$\footnote{If we consider the case of full vocabulary, i.e. countable predicates, countable functions symbols, and countable constant symbols, then such set is known to be $\Sigma_1$-complete (\cite{mn,bn}). However, since we are working on a language having only predicates, the proof given in \cite{mn} cannot be applied: in every case, the set of all theorems of L$\forall$ is at least recursively enumerable, and hence it belongs to $\Sigma_1$.}, then L$\forall$ cannot be complete w.r.t. the finite models of $\mathcal{A}$.

We conclude that if L is an axiomatic extension of MTL$_\Delta$ such that TAUT$_\text{L}$ is decidable, then the fSCC fails to holds, for L$\forall$.
\end{proof}
\paragraph{Acknowledgements}
The author would like to thank to Félix Bou for pointing out an error in the previous version of the article, and for the references \cite{bou,bn}.

Other thanks are due to Petr Cintula and Yuri Gurevich, for some informations about the results of, respectively, \cite{lcomp} and \cite{bgg}.
\bibliography{trakbib}
\bibliographystyle{amsplain}
\end{document}